\documentclass[11pt]{amsart}

\usepackage{amsmath}
\usepackage{amssymb}
\usepackage{amscd}
\usepackage{color}

\begin{document}

\title{On Cline's Formula for some certain elements in a ring}

\author{Orhan G\"{u}rg\"{u}n}
\address{Orhan G\"{u}rg\"{u}n, Department of Mathematics, Ankara University, Turkey}
\email{orhangurgun@gmail.com}

\date{\empty}
\date{}
\newtheorem{thm}{Theorem}[section]
\newtheorem{lem}[thm]{Lemma}
\newtheorem{prop}[thm]{Proposition}
\newtheorem{cor}[thm]{Corollary}
\newtheorem{exs}[thm]{Examples}
\newtheorem{defn}[thm]{Definition}
\newtheorem{nota}{Notation}
\newtheorem{rem}[thm]{Remark}
\newtheorem{ex}[thm]{Example}

\maketitle

\begin{abstract} In \cite{C, LCC}, it is proven that
if an element $ab$ in a ring is (generalized) Drazin invertible,
then so is $ba$.  In this paper,  we give a new and short proof of
it in an effective manner. In particular, we show that if $ab$ is
strongly clean, then so is $ba$. Consequently, we see that if
$1-ab$ is strongly clean, then so is $1-ba$. Also, some
characterizations are obtained for some certain elements in a
corner ring. It is shown that for an idempotent $e$ and any
arbitrary element $a$ in a ring, $ea+1-e$ is Drazin invertible if
and only if $eae$ is Drazin invertible.
 \\[+2mm]
\end{abstract}

\section{Introduction}

Throughout this paper, all rings are associative with identity.
The notation $U(R)$ denotes the group of units of $R$, $J(R)$
denotes the Jacobson radical of $R$ and $Nil(R)$ denotes the set
of all nilpotent elements of $R$. The \emph{commutant} and
\emph{double commutant} of an element $a$ in a ring $R$ are
defined by $comm_R(a)=\{x\in R~|~xa=ax\}$, $comm_R^2(a)=\{x\in
R~|~xy=yx~\mbox{for all}~y\in comm_R(a)\}$, respectively. We
simply use $comm(a)$ and $comm^2(a)$. If $R^{qnil}=\{a\in
R~|~1+ax\in U(R)~\mbox{for every}~x\in comm(a)\}$ and $a\in
R^{qnil}$,  then $a$ is said to be \emph{quasinilpotent} \cite{H}.
Set $J^{\#}(R)=\{ x\in R~\mid~\exists ~n\in {\Bbb N}~\mbox{such
that}~x^n\in J(R)\}$. Clearly, $J(R)\subseteq J^{\#}(R) \subseteq
R^{qnil}$.

Now we recall some definitions that are used in the next sections.
Following \cite{H}, an element $a\in R$ is called
\emph{quasipolar} provided that there exists $p^2=p\in comm^2(a)$
such that $a+p\in U(R)$ and $ap\in R^{qnil}$. Following
\cite{WC3}, an element $a\in R$ is called \emph{pseudopolar}
provided that there exists $p^2=p\in comm^2(a)$ such that $a+p\in
U(R)$ and $ap\in J^{\#}(R)$. An element $a\in R$ is said to be
\emph{strongly $\pi$-regular} if $a^n\in a^{n+1}R\cap Ra^{n+1}$
for some $n\in \Bbb N$. An element $a\in R$ is said to be
\emph{strongly clean} if there exists $e^2=e\in R$ such that
$a-e\in U(R)$ and $e\in comm(a)$ (\cite{N2}). An element $a$ of
$R$ is \emph{(respectively, pseudo, generalized) Drazin
invertible} (\cite{KP}) \cite{D, K} in case there is an element
$b\in R$ satisfying $ab^2 = b$, $b\in comm^2(a)$ and
(respectively, $a^2b-a\in J^{\#}(R)$, $a^2b-a\in R^{qnil}$)
$a^2b-a\in Nil(R)$. Such $b$, if it exists, is unique; it is
called a \emph{(respectively, pseudo, generalized) Drazin inverse}
of $a$ and will be denoted by (respectively, $a^{pD}$, $a^{gD}$)
$a^D$. Further, an element $a\in R$ is said to be \emph{group
invertible (or strongly regular)} if there exists $b\in R$ such
that $b\in comm(a)$(or equivalently, $b\in comm^2(a)$), $ab^2=b$
and $a^2b=a$. In this case, we say that $b$ is group inverse of
$a$ and will be denoted by $a^{\sharp}$.

Koliha showed that $a\in R$ is Drazin invertible if and only if
$a$ is strongly $\pi$-regular (\cite[Lemma 2.1]{K}). Koliha and
Patricio proved that $a\in R$ is generalized Drazin invertible if
and only if $a$ is quasipolar (\cite[Theorem 4.2]{KP}). Wang and
Chen showed that $a\in R$ is pseudo Drazin invertible if and only
if $a$ is pseudopolar (\cite[Theorem 3.2]{WC3}).

Cline \cite{C} showed that if $ab$ is Drazin invertible, then so
is $ba$, and $(ba)^D=b\big((ab)^D\big)^2a$. This equality is known
as \emph{Cline's formula}. Recently, Liao et al. \cite{LCC} showed
that if $ab$ is generalized Drazin invertible, then so is $ba$.
Also, Wang and Chen \cite{WC3} proved that if $ab$ is pseudo
Drazin invertible, then so is $ba$. In this note, we present a new
and short proof of these facts. We also see that if $ab$ is
strongly clean, then so is $ba$. By using this result, we derive
that if $1-ab$ is strongly clean, then so is $1-ba$. In Section
$3$, elements in a corner ring are studied. For example, we show
that $eae$ is strongly $\pi$-regular in $eRe$ if and only if
$ea+1-e$ is strongly $\pi$-regular if and only if $ae+1-e$ is
strongly $\pi$-regular. Moreover, we also see that if $1-ab$ is
pseudo Drazin invertible, then so is $1-ba$.

The following remarks summarize some results in \cite{G} so that
they can be easily referenced in this paper.

\begin{rem}\cite[Theorem 2.12]{G} \label{new1} Let $R$ be a ring with $a\in R$. Then the following are equivalent.

\begin{enumerate}
   \item $a$ is strongly $\pi$-regular.
   \item $a=s+n$ where $s$ is strongly regular, $n$ is nilpotent and $sn=ns=0$.
\end{enumerate}
\end{rem}

\begin{rem}\cite[Corollary 2.15]{G}\label{new2} Let $R$ be a ring with $a\in R$. Then the following are equivalent.

\begin{enumerate}
    \item $a$ is quasipolar.
    \item $a=s+q$ where $s$ is strongly regular, $s\in comm^2(a)$, $q\in R^{qnil}$ and $sq=qs=0$.
\end{enumerate}
\end{rem}

\begin{rem}\cite[Theorem 4.1]{G}\label{new3} Let $R$ be a ring with $a\in R$. Then the following are equivalent.

\begin{enumerate}
    \item $a$ is pseudopolar.
    \item $a=s+q$ where $s$ is strongly regular, $s\in comm^2(a)$, $q\in J^{\#}(R)$ and $sq=qs=0$.
\end{enumerate}
\end{rem}

\section{ Main Results }

We start with several auxiliary lemmas which are crucial in the
proof of the main results.

\begin{lem}\label{Jacobson}(Jacobson's Lemma) Let $R$ be a ring and $a,b\in R$. If
$1+ab$ is invertible, then so is $1+ba$ and
$(1+ba)^{-1}=1-b(1+ab)^{-1}a$.
\end{lem}

\begin{lem}\label{lem1} Let $R$ be a ring and $a\in R$. If $a=s+n$ where $s$ is strongly regular, $n$ is nilpotent and
$sn=ns=0$, then $a^{D}=s^{\sharp}$.
\end{lem}

\begin{proof} Write $s^{\sharp}:=y$. Then we have $y\in comm^2(s)$,
$sy^2=y$ and $ys^2=s$. Clearly, $y\in comm(a)$. Note that
$ay^2=as^2y^4=s^3y^4=y$ and $a^2y-a=a^2y^2s-a=s^3y^2-a=s-a=-n\in
Nil(R)$. This implies that $a^{D}=s^{\sharp}$.
\end{proof}

\begin{lem}\label{lem2} Let $R$ be a ring and $a\in R$. If $a=s+q$ where $s$ is strongly regular,
$s\in comm^2(a)$, $q\in R^{qnil}$ and $sq=qs=0$, then
$a^{gD}=s^{\sharp}$.
\end{lem}

\begin{proof} Similar to the proof of Lemma~\ref{lem1}.
\end{proof}

\begin{lem}\label{lem3} Let $R$ be a ring and $a\in R$. If $a=s+q$ where $s$ is strongly regular,
$s\in comm^2(a)$, $q\in J^{\#}(R)$ and $sq=qs=0$, then
$a^{pD}=s^{\sharp}$.
\end{lem}

\begin{proof} The proof is similar to the proof of Lemma~\ref{lem1}.
\end{proof}

We now prove

\begin{thm} \label{strclean} Let $R$ be a ring and $a,b\in R$. If
$ab$ is strongly clean, then so is $ba$.
\end{thm}

\begin{proof} Assume that $ab$ is strongly clean. Then there
exist $e^2=e\in comm(ab)$ and $u\in U(R)$ such that $ab=e+u$. Set
$f=bu^{-1}(1-e)a$, $g=1-f$ and $-v=g-ba=1-b[1+u^{-1}(1-e)]a$. This
implies that

$
\begin{array}{ll}
  f^2  & = bu^{-1}(1-e)abu^{-1}(1-e)a \\
       & = bu^{-1}(1-e)uu^{-1}(1-e)a \\
       & = bu^{-1}(1-e)a\\
       & = f\\
\end{array}
$

\noindent because $(1-e)ab=(1-e)u$, and so $g^2=g$. Moreover, we
have

$
\begin{array}{ll}
  1-ab[1+u^{-1}(1-e)]  & =1-ab-ab(1-e)u^{-1} \\
       & = 1-ab-(1-e)\\
       & = -u\in U(R).\\
\end{array}
$

\noindent In view of Lemma~\ref{Jacobson}, we see that $v\in
U(R)$. Since $fba=bu^{-1}(1-e)aba=b(1-e)a$ and
$baf=babu^{-1}(1-e)a=b(1-e)a$, we get $gv=vg$, also $ba=g+v$. This
gives that $ba$ is strongly clean. So the proof is completed.
\end{proof}



The following result is a direct consequence of
Theorem~\ref{strclean}.

\begin{cor}\label{cleanJacobson} Let $R$ be a ring and $a,b\in R$. If
$1-ab$ is strongly clean, then so is $1-ba$.
\end{cor}

\begin{proof} It is well known that $1-x$ is strongly clean if and
only if $x$ is strongly clean. Assume that $1-ab$ is strongly
clean. Then $ab$ is strongly clean. By Theorem~\ref{strclean}, we
see that $ba$ is strongly clean. So $1-ba$ is strongly
clean.
\end{proof}

\begin{thm} \label{strpiregular} Let $R$ be a ring and $a,b\in R$. If
$ab$ is Drazin invertible, then so is $ba$ and
$(ba)^D=b\big((ab)^D\big)^2a$.
\end{thm}

\begin{proof} By \cite[Theorem 4]{D}, there exist $c\in R$ and $m\in \Bbb N$ such that $abc=cab$,
$(ab)^m=(ab)^{m+1}c$ and $c=c^2ab$. Set $s=bcaba$ and $n=ba-s$.
This implies that $sn=ns=0$ and $e=bca$ is an idempotent. We
observe that
$s^2(bc^2a)=(bcaba)(bcaba)(bc^2a)=bc^4(ab)^4a=bcaba=s$ because
$c=c^2ab$. It is easy to check that $s^{\sharp}=bc^2a$. Note that
$n^{m+1}=(ba-s)^{m+1}=(ba-eba)^{m+1}=(ba)^{m+1}-bca(ba)^{m+1}=b[(ab)^m-c(ab)^{m+1}]a=0$
since $(ab)^m=(ab)^{m+1}c$. Then $ba$ is strongly $\pi$-regular by
Remark~\ref{new1}. Hence $ba$ is Drazin invertible. By
Lemma~\ref{lem1}, $(ba)^D=b\big((ab)^D\big)^2a$.
\end{proof}

\begin{thm} \label{quasipolar} Let $R$ be a ring and $a,b\in R$. If
$ab$ is generalized Drazin inverse, then so is $ba$ and
$(ba)^{gD}=b\big((ab)^{gD}\big)^2a$.
\end{thm}

\begin{proof} Suppose $ab$ is generalized Drazin invertible. Then there exists $c\in comm^2(ab)$ such that
$abc^2=c$ and $(ab)^2c-ab\in R^{qnil}$. Set $t=(ab)^2c-ab$,
$s=bcaba$ and $q=ba-s$. As in the proof of
Theorem~\ref{strpiregular}, we see that $sq=qs=0$, $s$ is strongly
regular and $s^{\sharp}=bc^2a$. Let $xq=qx$ for some $x\in R$. We
observe that
$ax^2b\big(ab-(ab)^2c\big)=ax^2\big(ba-s\big)b=ax^2qb=aqx^2b=ax^2b\big(ab-(ab)^2c\big)$
and so $ax^2b\in comm(t)$. Since $t\in R^{qnil}$, we have
$1-ax^2bt\in U(R)$, and so $1-x^2bta\in U(R)$ by
Lemma~\ref{Jacobson}. Further, $(1+xq)(1-xq)=1-x^2q^2=1-x^2bta\in
U(R)$, and so $1+xq\in U(R)$. That is, $q\in R^{qnil}$. Now let
$yba=bay$ for some $y\in R$. It is easy to see that $ayb\in
comm(ab)$, and so $cayb=aybc$ because $c\in comm^2(ab)$.
Multiplying by $b$ on the left and by $a$ on the right yields
$sy=babcay=bcabay=baybca=ybabca=ys$. Hence $s\in comm^2(ba)$. By
Remark~\ref{new2}, $ba$ is quasipolar. Thus $ba$ is generalized
Drazin invertible. So we have $(ba)^{gD}=b\big((ab)^{gD}\big)^2a$
by Lemma~\ref{lem2}.
\end{proof}

\begin{thm} \label{pseudopolar} Let $R$ be a ring with $a,b\in R$. If
$ab$ is pseudo Drazin invertible, then so is $ba$ and
$(ba)^{pD}=b\big((ab)^{pD}\big)^2a$.
\end{thm}

\begin{proof} Assume $ab$ is pseudo Drazin invertible. Then there exists $c\in comm^2(ab)$ such that
$abc^2=c$ and $(ab)^2c-ab\in J^{\#}(R)$. Set $t=(ab)^2c-ab$,
$s=bcaba$ and $q=ba-s$. Similar to the proof of
Theorem~\ref{quasipolar}, we see that $sq=qs=0$, $s$ is strongly
regular, $s\in comm^2(ba)$ and $s^{\sharp}=bc^2a$. Now we show
that $q\in J^{\#}(R)$. Since $t\in J^{\#}(R)$, we have $t^n\in
J(R)$ for some $n\in \Bbb N$. We observe that
$q^{n+1}=[(1-bca)ba]^{n+1}=(1-bca)(ba)^{n+1}=bt^na\in J(R)$
because $1-bca$ is an idempotent. Hence, by Remark~\ref{new3},
$ba$ is pseudopolar. That is, $ba$ is pseudo Drazin invertible and
$(ba)^{pD}=b\big((ab)^{pD}\big)^2a$ by Lemma~\ref{lem3}.
\end{proof}


\section{Elements in Corner Rings}

 Let $a$, $e^2=e\in R$. In \cite{LM} showed
that $eae$ is unit regular in $eRe$ iff $eae+1-e$ is unit regular
in $R$. Then Chen \cite{Chen} proved that $eae$ is unit regular in
$eRe$ iff $ea+1-e$ is unit regular in $R$ iff $ae+1-e$ is unit
regular in $R$. Now we extend this result to some certain elements
in ring.

\begin{thm} Let $R$ be a ring with $a$, $e^2=e\in R$. Then the
following statements are equivalent.
\begin{enumerate}
    \item $eae$ is strongly clean in $eRe$.
    \item $ae + 1 - e$ is strongly clean in $R$.
    \item $ea + 1 - e$ is strongly clean in $R$.
\end{enumerate}
\end{thm}

\begin{proof} $(1)\Rightarrow (2)$ Suppose $eae$ is strongly
clean. Then $e-eae$ is also strongly clean in $eRe$. By
\cite[Theorem 1.2.13]{Di}, we have $e-eae$ is strongly clean in
$R$. This gives $1-e+eae=1-e(1-ae)$ is strongly clean. By
Corollary~\ref{cleanJacobson}, we get $1-(1-ae)e=1-e+ae$ is
strongly clean in $R$.

$(2)\Rightarrow (1)$ Assume $ae + 1 - e=1-(1-ae)e$ is strongly
clean. This gives $1-e(1-ae)=1-e+eae$ is also strongly clean by
Corollary~\ref{cleanJacobson}. Then $1-(1-e+eae)=e-eae$ is
strongly clean. In view of \cite[Theorem 1.2.13]{Di}, we have
$e-eae$ is strongly clean in $eRe$. It follows that
$e-(e-eae)=eae$ is strongly clean. So holds $(1)$.

$(2)\Rightarrow (3)$ If $ae + 1 - e=1-(1-ae)e$ is strongly clean,
then $1-e(1-ae)=1-e+eae=1-(1-ea)e$ is strongly clean by
Corollary~\ref{cleanJacobson}. This gives that $1-e(1-ea)=1-e+ea$
is strongly clean again by Corollary~\ref{cleanJacobson}.

$(3)\Rightarrow (2)$ is symmetric.
\end{proof}

By \cite[Theorem 3.6]{CMP} (see also \cite{LP, ZCC}), if $1-xy$ is
strongly $\pi$-regular, then so is $1-yx$. By using this property,
we obtain the following result.

\begin{thm}\label{sregular} Let $R$ be a ring with $a$, $e^2=e\in R$. Then the
following statements are equivalent.
\begin{enumerate}
    \item $eae$ is strongly $\pi$-regular in $eRe$.
    \item $ae + 1 - e$ is strongly $\pi$-regular in $R$.
    \item $ea + 1 - e$ is strongly $\pi$-regular in $R$.
\end{enumerate}
\end{thm}

\begin{proof}
$(1)\Rightarrow(2)$ Assume $eae=s+n$ where $s\in eRe$ is strongly
regular, $n\in Nil(eRe)$ and $sn=ns=0$. Then there exists $y\in
eRe$ such that $y\in comm^2(s)$ and $s=s^2y$. By using this, we
get $(s+1-e)(s+1-e)(y+1-e)=(s^2+1-e)(y+1-e)=s^2y+1-e=s+1-e$. So
$s+1-e$ is strongly regular in $R$ and $(s+1-e)n=n(s+1-e)=0$.
Hence, by Remark~\ref{new1}, $eae+1-e=s+1-e+n$ is strongly
$\pi$-regular in $R$. Since $eae + 1 - e=1-e(1-ae)$, we have $ae +
1 - e$ is strongly $\pi$-regular in $R$.

$(2)\Rightarrow(1)$ Suppose $1-e+ae$ is strongly $\pi$-regular in
$R$. As $1-e+ae=1-(1-ae)e$, we have $\alpha=eae + 1 - e$ is
strongly $\pi$-regular in $R$. Then there exists $b\in R$ such
that $\alpha b^2 = b$, $b\in comm^2(\alpha)$ and
$\alpha^2b-\alpha\in Nil(R)$. Since $\alpha^2b-\alpha\in Nil(R)$,
there exists $n\in \Bbb N$ such that
$(\alpha^2b-\alpha)^n=\alpha^{n+1}b-\alpha^{n}=0$, and so
$\alpha^n=\alpha^{n+1}b$. Note that $\alpha^{n}=(eae)^n+1-e$,
$\alpha^{n+1}=(eae)^{n+1}+1-e$ and $eb=be$ because $b\in
comm^2(\alpha)$ and $e\alpha=\alpha e$. We get
$(eae)^n=(eae)^{n+1}b+b(1-e)-(1-e)$. Multiplying by $e$, we see
that $(eae)^n=(eae)^{n+1}ebe$. It can be verified that $ebe\in
comm(eae)$. Thus $eae$ is strongly $\pi$-regular in $eRe$.

$(1)\Leftrightarrow(3)$ is symmetric.
\end{proof}

By \cite[Theorem 2.3]{ZCC}, if $1-xy$ is quasipolar, then so is
$1-yx$. Then we have the following result.

\begin{thm}\label{qpolar} Let $R$ be a ring with $a$, $e^2=e\in R$.
Consider the following statements.
\begin{enumerate}
    \item $ae + 1 - e$ is quasipolar in $R$.
    \item $ea + 1 - e$ is quasipolar in $R$.
    \item $eae$ is quasipolar in $eRe$.
\end{enumerate}

\noindent Then we have $(1)\Leftrightarrow(2)\Rightarrow(3)$. If
$e$ is central, then $(3)\Rightarrow(1)$.
\end{thm}

\begin{proof} $(1)\Rightarrow(3)$ Since $ae + 1 - e=1-(1-ae)e$ is
quasipolar, we have $1-e(1-ae)=1-e+eae$ is quasipolar. Then
$eae+1-e=s+q$ where $s\in R$ is strongly regular, $s\in
comm^2(eae+1-e)$, $q\in R^{qnil}$ and $sq=qs=0$. Since $s$ is
strongly regular, there exists $y\in comm^2(s)$ such that
$s=s^2y$. Write $eae=ese+eqe$. Hence $es=se$, $ye=ey$ and $eq=qe$
because $s\in comm^2(eae+1-e)$, $sq=qs$ and $y\in comm^2(s)$. By
using this, we see that $(ese)(eqe)=esqe=0=(eqe)(ese)$ and
$(ese)(ese)(eye)=eseseye=es^2ye=ese$; that is, $ese$ is strongly
regular in $eRe$.

\textbf{Claim 1.} $eqe\in (eRe)^{qnil}$.

\emph{Proof.} Let $(exe)(eqe)=(eqe)(exe)$ for some $x\in R$. Now
we show that $e+exeqe\in U(eRe)$. Then $exeq=qexe$ because
$eq=qe$. Since $q\in R^{qnil}$, we have $1+exeq\in U(R)$. Set
$v=(1+exeq)^{-1}$. We observe that
$(e+exeqe)(eve)=(e+exeq)(ve)=e(1+exeq)ve=e$. Similarly, we show
that $(eve)(e+exeq)=e$. Therefore $eqe\in (eRe)^{qnil}$, as
desired.

\textbf{Claim 2.} $ese\in comm^2(eae)$.

\emph{Proof.} Let $(exe)(eae)=(eae)(exe)$ for some $x\in R$. We
show that $(exe)(ese)=(ese)(exe)$. Since $s\in comm^2(eae+1-e)$
and $exe\in comm(eae+1-e)$, we have $exes=sexe$. It follows that
$(exe)(ese)=(ese)(exe)$. That is, $ese\in comm^2(eae)$.

\noindent Hence $eae\in eRe$ is quasipolar by Remark~\ref{new2}.

$(3)\Rightarrow(1)$ Suppose $eae+p=q\in U(eRe)$ where $p^2=p\in
comm^2(eae)$ and $eaep\in (eRe)^{qnil}$. Write $eae+1-e+p=q+1-e$.
Since $q\in U(eRe)$, $qv=vq=e$ for some $v\in eRe$. Then
$(q+1-e)(v+1-e)=qv+1-e=1$ and so $q+1-e\in U(R)$. Further, we have
$(eae+1-e)p=eaep\in R^{qnil}$ because $(eRe)^{qnil}=eRe\cap
R^{qnil}$ by \cite[Lemma 3.5]{YC}. As $e$ is central, it is easy
to see that $p\in comm^2(eae+1-e)$. Since $eae + 1 - e=1-e(1-ae)$,
we have $ae + 1 - e$ is quasipolar in $R$.

$(1)\Rightarrow (2)$ If $ae + 1 - e=1-(1-ae)e$ is quasipolar, then
$1-e(1-ae)=1-e+eae=1-(1-ea)e$ is quasipolar. This gives that
$1-e(1-ea)=1-e+ea$ is quasipolar.

$(2)\Rightarrow (1)$ is symmetric.
\end{proof}

Following \cite{ZCC} and \cite{LP}, we prove the following result.

\begin{thm} \label{jacpseudo} Suppose  $\alpha:= 1-ab\in R$ is pseudo Drazin invertible with index $k$, pseudo Drazin inverse $\alpha^{pD}$, and
strongly spectral idempotent $e := 1-\alpha^{pD}\alpha$, and let
$u:=1-\alpha e$. Then $\beta:= 1-ba$ is also pseudo Drazin
invertible with index $k$, and has pseudo Drazin inverse

$$ \beta^{pD}=(1-beu^{-1}a) + b\alpha^{pD}a.$$

\noindent  The strongly spectral idempotent for $\beta$ is
$beu^{-1}a$.
\end{thm}

\begin{proof} We know that $e\in comm^2(\alpha)$,
$\alpha^ke\in J(R)$ and $\alpha+e\in U(R)$. Then $u\in U(R)$ and
write $f:=beu^{-1}a$. Similar to the proof of \cite[Lemma
2.3]{LP}, we see that $\beta^kf\in J(R)$. Moreover, it is easy to
check that $abeu^{-1}=e$, and so we have $fb=be$ and $ea=af$. So
$f^2=f$. Note that $\beta+f=1-b(1-eu^{-1})a\in U(R)$ if an only if
$\alpha+e=1-(1-eu^{-1})ab\in U(R)$ by Lemma~\ref{Jacobson}. Then
$\beta+f\in U(R)$ because $\alpha+e\in U(R)$. To prove $f\in
comm^2(\beta)$, let $x\beta=\beta x$. We observe that $xba=bax$
and $axb\in comm(\alpha)$, and so $axbe=eaxb$ and
$axbu^{-1}=u^{-1}axb$ since $e\in comm^2(\alpha)$. Hence we have
$fxf=beu^{-1}axbeu^{-1}a=baxbeu^{-1}eu^{-1}a=xbabeu^{-1}eu^{-1}a=xf$
and
$fxf=beu^{-1}axbeu^{-1}a=beu^{-1}eu^{-1}axba=beu^{-1}eu^{-1}abax=fx$.
That is, $xf=fx$. Thus $f^2=f\in comm^2(\beta)$. As is well known,
$\beta^{pD}=(\beta+f)^{-1}(1-f)$. By Lemma~\ref{Jacobson}, it is
easy to see that $(\beta+f)^{-1}=1-b(\alpha+e)^{-1}(eu^{-1}-1)a$.
Therefore,

$
\begin{array}{ll}
  \beta^{pD}  & =(1-b(\alpha+e)^{-1}(eu^{-1}-1)a)(1-f) \\
       & = 1-f-b(\alpha+e)^{-1}(eu^{-1}-1)(a-af) \\
       & = 1-f-b(\alpha+e)^{-1}(eu^{-1}-1)(a-ea)\\
       & = 1-f-b(\alpha+e)^{-1}(eu^{-1}-1)(1-e)a\\
       & = 1-f+b(\alpha+e)^{-1}(1-e)a\\
       & = 1-f+b\alpha^{pD}a.\\
\end{array}
$

\end{proof}

Let $e\in R$ be an idempotent and $a\in R$. Wang and Chen
\cite{WC3} showed that if $eae+1-e$ is pseudopolar in $R$, then
$eae$ is pseudopolar in $eRe$. We extend this result as follows.

\begin{thm}\label{ppolar} Let $R$ be a ring with $a$, $e^2=e\in R$.
Consider the following statements.
\begin{enumerate}
    \item $ae + 1 - e$ is pseudopolar in $R$.
    \item $ea + 1 - e$ is pseudopolar in $R$.
    \item $eae$ is pseudopolar in $eRe$.
\end{enumerate}

\noindent Then we have $(1)\Leftrightarrow(2)\Rightarrow(3)$. If
$e$ is central, then $(3)\Rightarrow(1)$.
\end{thm}

\begin{proof} $(1)\Rightarrow(3)$ Since $ae + 1 - e=1-(1-ae)e$ is
pseudopolar, we have $1-e(1-ae)=1-e+eae$ is pseudopolar by
Theorem~\ref{jacpseudo}. Then $eae+1-e=s+q$ where $s\in R$ is
strongly regular, $s\in comm^2(eae+1-e)$, $q\in J^{\#}(R)$ and
$sq=qs=0$. Write $eae=ese+eqe$. Similar to the proof of
Theorem~\ref{qpolar}, we show that $ese$ is strongly regular in
$eRe$, $ese\in comm^2(eae)$ and $(ese)(eqe)=(eqe)(ese)=0$. As
$q\in J^{\#}(R)$, we have $q^n\in J(R)$ for some $n\in \Bbb N$. It
is easy to check that $(eqe)^n=eq^ne\in eJ(R)e=J(eRe)$ because
$eq=qe$. So $eqe\in J^{\#}(eRe)$. By Remark~\ref{new3}, $eae$ is
pseudopolar in $eRe$.

$(3)\Rightarrow(1)$ Assume $eae+p=q\in U(eRe)$ where $p^2=p\in
comm^2(eae)$ and $eaep\in J^{\#}(eRe)$. Then there exists $n\in
\Bbb N$ such that $(eaep)^n=(eae)^np\in J(R)$. Write
$eae+1-e+p=q+1-e$. As in the proof of Theorem~\ref{qpolar}, we get
$q+1-e\in U(R)$ and $p\in comm^2(eae+1-e)$. Further, we have
$\big((eae+1-e)p\big)^n=(eae+1-e)^np=\big((eae)^n+1-e\big)p=(eae)^np\in
J(R)$. That is, $(eae+1-e)p\in J^{\#}(R)$. Since $eae + 1 -
e=1-e(1-ae)$, we have $ae + 1 - e$ is pseudopolar in $R$.

$(1)\Rightarrow (2)$ If $ae + 1 - e=1-(1-ae)e$ is pseudopolar,
then $1-e(1-ae)=1-e+eae=1-(1-ea)e$ is pseudopolar. Hence we have
$1-e(1-ea)=1-e+ea$ is pseudopolar.

$(2)\Rightarrow (1)$ is symmetric.
\end{proof}

\end{document}